\documentclass{article}
\usepackage{graphicx} 
\usepackage{amsmath}
\usepackage{setspace}
\usepackage{amssymb}
\usepackage{amsthm}
\IfFileExists{yhmath.sty}{\usepackage{yhmath}}{} 
\usepackage{url}
\usepackage{amsfonts}
\usepackage[dvipsnames]{xcolor}
\usepackage{geometry}
\usepackage{tcolorbox}
\IfFileExists{algorithm.sty}{\usepackage{algorithm}}{
  \usepackage{float}
  \floatstyle{ruled}
  \newfloat{algorithm}{htbp}{loa}
  \floatname{algorithm}{Algorithm}
}
\usepackage{tikz}
\tikzset{
  leaf/.style={inner sep=0pt,font=\LARGE, outer sep=0pt},
  node/.style={inner sep=0pt,font=\normalsize, outer sep=0pt}
}

\tikzset{
  C/.style={
   minimum size=4mm, font=\Large,inner sep=0pt, outer sep=0pt,
   execute at begin node=\ensuremath{\times},
    path picture={
      \draw (path picture bounding box.center) circle (1.5mm);
    }
  }
}
\tikzset{
  S/.style={
    draw, rectangle, minimum size=5mm, inner sep=0pt, outer sep=0pt
  }
}
\tikzset{
  root/.style={
    node,
    execute at begin node=\ensuremath{\times}
  }
}
\tikzset{
  D/.style={
    leaf,
    execute at begin node=\ensuremath{\diamond}
  }
}
\tikzset{
  O/.style={
    leaf,inner sep=0pt,outer sep=0pt,
    execute at begin node=\ensuremath{\circ}
  }
}

\title{Generalized Feynman Graph Expansions for Lévy-Driven Pseudo-Differential Evolution Equations}
\author{Iskander O. Hamdiken$^1$\,\,\,\, Smii Boubaker$^1$\footnote{
corresponding author, and e-mail: boubaker@kfupm.edu.sa }\,\,\, 
\\$^1$ King Fahd University of Petroleum and Minerals\\Department of Mathematics\\Dhahran 31261, Saudi Arabia\\
 }

\theoremstyle{plain}
\newtheorem{theorem}{Theorem}[section]
\newtheorem{proposition}[theorem]{Proposition}
\newtheorem{lemma}[theorem]{Lemma}
\newtheorem{corollary}[theorem]{Corollary}

\theoremstyle{definition} 
\newtheorem{definition}[theorem]{Definition}
\newtheorem{assumption}[theorem]{Assumption}
\newtheorem{example}[theorem]{Example}

\theoremstyle{remark} 
\newtheorem{remark}[theorem]{Remark}

\date{\today}
\begin{document}

\onehalfspacing
\maketitle

\newcommand{\R}{\mathbb{R}}
\newcommand{\C}{\mathbb{C}}
\newcommand{\F}{\mathcal{F}}
\newcommand{\Sspace}{\mathcal{S}}
\newcommand{\Ssp}{\mathcal{S}'}
\newcommand{\N}{\mathbb{N}}
\newcommand{\Snorm}[3]{\left\|#1\right\|_{#2,#3}}
\newcommand{\pa}{\partial}
\newcommand{\dd}{\,\mathrm{d}}

\begin{center}
    \textbf{Abstract}
\end{center}

We study a class of Cauchy problems driven by pseudo-differential operators with Lévy-type symbols. We extend earlier results from exponential initial data to a broader set of initial conditions. We demonstrate the existence and uniqueness of the solution, which has an explicit convolution representation with the associated kernel. We then derive an asymptotic expansion of the solution, with coefficients expressed in terms of moments and, after taking logarithms, in terms of cumulants. These expansions are organized through generalized Feynman graph representations, providing a clear combinatorial description of both the solution and its logarithm. We also establish a Borel summability result under suitable conditions and recover the classical exponential case as a special instance. 
Finally, we illustrate the convolution kernel and solution through a one-dimensional Lévy-Gaussian example.


\section{Introduction}
Evolution equations model physical systems and processes evolving in time, describing the evolution of complex systems arising in physics, probability theory, and applied mathematics. Their mathematical formulation captures the dynamics of state function associated with various fields such as diffusion processes, quantum field theory, and astrophysics, see, e.g., \cite{lattanzi, Salmhofer, volpe}. 

In many situations, the evolution dynamics are represented by partial differential equations whose generator is given by an explicit representation, such as the Laplacian operator which appears in the Heat equation, see, e.g., \cite{cannon}. 

An important class of models is constructed by replacing differential operators by pseudo-differential operators characterized through their Fourier symbols, see, e.g., \cite{taylor}. Such operators can arise in the system evolving nonlocal dynamics, Lévy processes, anomalous diffusions and in general stochastic systems with jumps, see, e.g., \cite{Alb1, zabczyk, Sato}. 

In particular, asymptotic methods are effective tools for investigating the behavior of the solution beyond the classical semigroup techniques, see, e.g., \cite{Alb2, Alb4, Smii21}. Such expansions are usually divergent; however, they can be interpreted through suitable resummation techniques, such as Borel summability, see, e.g., \cite{Alb3, main,sokal}. 

Graphical techniques inspired by Feynman diagrams play a central role in perturbative analysis and stochastic quantization, see, e.g., \cite{ParisiWu, Salmhofer}. These methods provide a combinatorial framework for organizing moments, cumulants, and iteration terms appearing in the asymptotic expansions. 

These techniques were adopted to stochastic differential equations driven by Lévy space-time noise, see, e.g., \cite{schwinger, main, peccati, Smii14,Smii21}. In particular, the graph expansion developed in \cite{main} established a representation of convolution semigroups associated with Lévy statistics, which connects analytic expansion with combinatorial graph structures.

Such representations are essential when logarithmic expansions and cumulants are studied and partitions and connected graphs are used. 

The present work continues this line of research by considering a class of Cauchy problems driven by a pseudo-differential operator with Lévy type symbols. Our main objective is to generalize the work done in \cite{main} by extending the class of initial conditions and by reconstructing both the analytical and graphical perturbative expansions. 

More precisely, we establish existence and uniqueness results for the corresponding evolution equation, and derive its explicit convolution representation of the solution kernel. We obtain asymptotic expansion for the solution in terms of moments, and after a logarithmic transformation, in terms of cumulants. We then express the obtained expansion using generalized Feynman graphs. 

We will also study the summability property of the expansion.  Under suitable assumptions on the Lévy symbol and the initial data, we prove Borel summability result that extends previous studies on Gaussian and Lévy noise, see, e.g., \cite{Alb3, Smii21}. This provides a rigorous interpretation of the formal asymptotic series and recovers the classical exponential initial condition studied in \cite{main} as a particular case. Furthermore, the graphical approach we develop clarifies the roles of the connected and disconnected contributions in the cumulant expansion and establishes a direct correspondence between analytic coefficients and graph theoretic structures, see, e.g., \cite{bollobas, oystein, serre}. 

The paper is organized as follows. In Section 2, we introduce the mathematical setting, notation, and preliminary results required throughout this work. Section 3 is devoted to the derivation of the asymptotic expansion of the solution . In Section 4, we give its representation through generalized Feynman graphs. Section 5 analyzes the formal logarithm of the solution and its graphical interpretation in terms of connected graph structures and partitions. In Section 6, we establish the Borel summability of the formal series obtained previously. Finally, in Section 7, we recover the exponential case and provide an exact benchmark and the direct and logarithmic truncation procedures.

\section{Assumptions and Problem Setting}

Let $\Sspace(\R^d)$ be the Schwartz space with
seminorms
\begin{equation}
\|u\|_{\alpha,N}:=\sup_{x\in\R^d}(1+|x|)^N\,|\pa^\alpha u(x)|,
\qquad\quad N\in\N_0,
\end{equation}
where $\alpha=(\alpha_1,\dots,\alpha_d)\in\N_0^d=\Big\{\N \cup \{0\}\Big\}^d$ is a multi-index with
\begin{equation}
\pa^\alpha:=\frac{\partial^{|\alpha|}}{\partial x_1^{\alpha_1}\cdots \partial x_d^{\alpha_d}},
\qquad |\alpha|:=\alpha_1+\cdots+\alpha_d, \qquad \alpha!=\alpha_1!\alpha_2!\cdots\alpha_d!\,.
\end{equation}

For $\alpha, \gamma\in \N_0^d$, we introduce the multinomial coefficient $\binom{\alpha}{\gamma}$ by 

\begin{equation}
\binom{\alpha}{\gamma}
  := \prod_{j=1}^d \binom{\alpha_j}{\gamma_j},\qquad \gamma\leq\alpha,
\end{equation}
where $\gamma\leq\alpha$ means $\gamma_j \le \alpha_j, \forall j$. We use the convention that $\binom{\alpha}{\gamma}=0$ whenever a component of $\gamma$ is negative or $\gamma_j>\alpha_j$ for some $j\in\{1,\cdots,d\}$.

Let $\Sspace'(\R^d)$ be the topological dual of $\Sspace(\R^d)$, i.e the space of tempered distributions on 
$\R^d$. For $u\in \mathcal S(\R^d)$ and $\xi\in \R^d$, we define the Fourier transform (resp. Fourier transform inverse) by
\begin{equation}
\widehat u(\xi):=\F (u)(\xi):=\int_{\R^d} e^{i\langle x,\xi\rangle}\,u(x)\,dx,
\qquad
u(x):=\F^{-1}(\widehat{u})(x):=(2\pi)^{-d}\int_{\R^d} e^{-i\langle x,\xi\rangle}\,\widehat{u}(\xi)\,d\xi.
\end{equation}

We extend $\F$ to tempered distributions $T\in\mathcal S'(\R^d)$ by duality:
\begin{equation}
\langle \widehat T,u\rangle := \langle T,\widehat u\rangle,
\qquad u\in\mathcal S(\R^d),
\end{equation}
(where $\langle\cdot,\cdot\rangle$ is the canonical pairing between $\mathcal S$ and $\mathcal S'$).
Moreover, for $u,v\in \mathcal S(\R^d)$ the Parseval identity holds:
\begin{equation}
\int_{\R^d} u(x)\,{v(x)}\,dx
=(2\pi)^{-d}
\int_{\R^d} \widehat u(\xi)\,{\widehat v(-\xi)}\,d\xi.
\end{equation}

In this work we focus on the following Cauchy problem

\begin{equation}\label{problem}
\begin{cases}
\partial_t\Phi_t(\phi)=\Psi(\nabla)\Phi_t(\phi), \qquad (t,\phi)\in \R_+\times \R^d, \\ \Phi_0(\phi)=f\!\big(\beta V(\phi)\big), \qquad\qquad\qquad \phi\in\R^d,\,\,\beta\in\R^*, 
\end{cases}
\end{equation}
where $f:\R\rightarrow \R$, and $V:\R^d\to\R$ is a polynomial of the form
\begin{equation}\label{Vdef}
V(\phi)=\sum_{p=0}^{\bar p}\big\langle \lambda^{(p)},\phi^{\otimes p}\big\rangle_p=\sum_{p=0}^{\bar p}\sum_{i_1,\dots,i_p=1}^d
\lambda^{(p)}_{i_1\cdots i_p}\,\phi_{i_1}\cdots \phi_{i_p},
\qquad \bar p\in\N.
\end{equation}

Here $\lambda^{(p)}$ are symmetric tensors of $p$-th degree over $\R^d$. 

Equivalently, we can write for $\alpha\in\N_0^d$
\begin{equation}\label{Vdef2}
V(\phi)=\sum_{|\alpha|\le \bar p} c_\alpha \phi^\alpha,\qquad
c_\alpha=\frac{|\alpha|!}{\alpha!}\,\lambda^{(|\alpha|)}_\alpha.
\end{equation}
{

$\Psi(\nabla)$ represents a pseudo-differential operator with symbol defined for $\xi$ belonging to a neighbourhood of zero by the following expansion 

\begin{equation}\label{pseudo}
    \Psi(\xi) = \sum_{n=0}^{\infty} \langle C^{(n)}, \xi^{\otimes n}\rangle 
= \sum_{n=0}^{\infty} \sum_{i_1,\ldots,i_n=1}^d C^{(n)}_{i_1 \cdots i_n} \, \xi_{i_1} \cdots \xi_{i_n},
\end{equation}
where $C^{(n)}=(C^{(n)}_{i_1\cdots i_n})$ are symmetric tensors of $n$-th degree over $\R^d$. 

We denote the Fourier symbol of $\Psi(\nabla)$ by $\psi$. Clearly, for $\xi\in\R^d$, $\psi(\xi)=\Psi(-i\xi)$. We assume that the symbol satisfies the following assumptions}
\begin{itemize}
\item[(P1)] $\xi\mapsto\Psi(-i\xi)$ is analytic in a neighborhood of $0\in\R^d$.

\item[(P2)] $\Re \Psi(-i\xi)\le c\log(1+\|\xi\|)$ for all $\xi\in\R^d$ and some $c>0$, with $\Re(z)$ being the real part of $z$.
\item[(P3)] $\forall\,\alpha\in\N_0^d\ \exists\,N_\alpha\in\N,\ b_\alpha>0:\ 
|\pa^\alpha\Psi(-i\xi)|\le b_\alpha(1+\|\xi\|)^{N_\alpha}\ \forall\xi\in\R^d.$
\end{itemize}

For every $t\geq 0$, we define the convolution kernel $\nu_t$ associated with the solution $\Phi_t$ of the Cauchy problem \eqref{problem} by
\begin{equation}\label{nu}
\nu_t:=\F^{-1}(e^{t\Psi(-i\xi)}).
\end{equation}

{Denote by $\mathcal O_M(\R^d)$ the space of Schwartz Multipliers, that is, the space of infinitely differentiable functions whose derivatives are polynomially bounded. 
}
The following result holds:  

\begin{theorem}\label{nutemp}
    For every $t\geq0$, the convolution kernel $\nu_t$ defines a tempered distribution, i.e, $\nu_t\in \mathcal S'(\R^d)$.
\end{theorem}

\begin{proof} By assumptions $(P1), (P2)$, and $(P3)$, the function $\R^d\ni\xi\mapsto e^{t\Psi(-i\xi)}$, $t\geq0$, is infinitely differentiable
with polynomially bounded derivatives. Therefore it belongs to $\mathcal O_M(\R^d)$, and by \cite{horvath}, Chap.4, Sec.11, Theorem 2, it defines a tempered distribution. Hence, by duality,
$\nu_t\in\mathcal S'(\R^d)$ satisfying
\begin{equation}\label{nuhat}
\widehat{\nu_t}(\xi)=e^{t\Psi(-i\xi)},\qquad \xi\in\R^d.
\end{equation}
\end{proof}

\begin{assumption}\label{assumptions} Throughout this work, we will assume the following: 

\begin{enumerate}
    \item\label{I} $f:\R\rightarrow \R$ is a function in the Schwartz space on $\R$ (i.e $f\in \mathcal S(\R)$).
    \item\label{II} $V:\R^d\to\R$ is a coercive polynomial, i.e, $|V(\phi)|\to\infty$ as $|\phi|\to\infty$.
    \item\label{III} The symbol of the pseudodifferential operator $\Psi(\nabla)$ is of Lévy type, that is 
    \begin{equation}\label{levysymbol}
\Psi(-i\xi) = i\langle a, \xi\rangle - \langle \xi, D\xi\rangle 
+ z \int_{\mathbb{R}^d \setminus \{0\}} \big(e^{i\langle \varphi,\xi\rangle} - 1\big)\, dr(\varphi), \qquad \xi\in\R^d.
\end{equation}
where $a \in \mathbb{R}^d$, $D\in \mathcal M_d(\mathbb R)$ symmetric positive semi-definite matrix,  { $z\geq0$ is the jump intensity,}
 and $r$ is a probability measure on $\mathbb{R}^d \setminus \{0\}$ with entire analytic Fourier transform.
\end{enumerate}
\end{assumption} 

Since the initial condition of the Cauchy problem \eqref{problem} is of the form $f(\beta V)$, the regularity and stability of such compositions in $\Sspace(\R^d)$ must be established. 

The following result holds

\begin{theorem}\label{fschwartz}
    For $\beta \in \R^*$, we have $f(\beta V)\in\mathcal S(\R^d)$. 
\end{theorem}

\begin{proof} 
The idea of the proof is based on the extension of \cite[Theorem 2.3]{galbis} through \cite[Remark 2.4]{galbis}. We first prove that the map $\phi\mapsto \beta V(\phi)$ is a Schwartz symbol on $\R^d$. For that, we fix $\beta\in\R^*$. By Assumption \ref{assumptions},
the polynomial $P=V^2$ is coercive. By
\cite[Section~2, Eq.~(2), and Corollary~3.4]{bajbar},
there exist $a,R,\ell>0$ such that
\begin{equation}
V(\phi)^2\geq a|\phi|^\ell,
\qquad |\phi|\geq R.
\end{equation}

Hence
\begin{equation}
|\beta V(\phi)|
\geq |\beta|\sqrt a\,|\phi|^{\ell/2},
\qquad |\phi|\geq R.
\end{equation}

Let now $c=|\beta|\sqrt a>0$ and $q=\ell/2>0$, and choose $K>1$ such that
\begin{equation}
K\geq\max\left\{R,\frac{2}{q},c^{-2/q}\right\}.
\end{equation}

For $|\phi|\geq K$, we have $c|\phi|^{q/2}\geq1$
and $1/K\leq q/2$. Hence
\begin{equation}
|\beta V(\phi)|
\geq c|\phi|^q
=\bigl(c|\phi|^{q/2}\bigr)|\phi|^{q/2}
\geq|\phi|^{q/2}
\geq|\phi|^{1/K},
\end{equation}
which proves the $d$-dimensional analogue of condition (ii) in \cite[Theorem 2.3]{galbis}.

For every multiindex $\alpha$, the derivative
$\partial^\alpha(\beta V)$ is a polynomial.
Therefore we can find constants $C_{\alpha,\beta},p_\alpha>0$ such that
\begin{equation}
|\partial^\alpha(\beta V)(\phi)|
\leq C_{\alpha,\beta}
     (1+|\beta V(\phi)|^2)^{p_\alpha},
\qquad \phi\in\R^d.
\end{equation}

The analogue growth conditions in the sufficient
part of \cite[Theorem~2.3 and Remark~2.4]{galbis} are satisfied. Since $f\in\mathcal S(\R)$ by
Assumption \ref{assumptions}, we obtain
$f(\beta V)\in\mathcal S(\R^d)$.
\end{proof}

For $\beta\in\R^*$,  and for fixed $V$ as in Assumptions \ref{assumptions}, we denote $f_\beta(\cdot)=f(\beta V(\cdot))$. The following result holds.

\begin{theorem}\label{existuniq}
Under assumptions \ref{assumptions}, the Cauchy problem \eqref{problem} has a unique solution \\$\Phi_t^\beta\in\mathcal S'(\R^d)$ given by
\begin{equation}\label{phisol}
\Phi_t^\beta(\phi)=(\nu_t*f_\beta)(\phi),\qquad \phi\in \R^d, \quad \beta\in\R^*
\end{equation}

\end{theorem}
\begin{proof}{
By Theorem \ref{fschwartz}, we have $f_\beta\in\mathcal S$. On the other hand by (P2)-(P3), the multiplier $e^{t\Psi(-i\xi)}$ and its time derivative have polynomially bounded derivatives. Thus $e^{t\Psi(-i\xi)}\widehat{f_\beta}$ defines a tempered distribution, with initial value $\widehat{f_\beta}$ and derivative $\Psi(-i\xi)e^{t\Psi(-i\xi)}\widehat{f_\beta}$. Applying the Fourier transform proves existence and equality \eqref{phisol} holds. For uniqueness, the Fourier transform of the difference of two solutions satisfies the homogeneous multiplier equation with zero initial value. Multiplying by $e^{-t\Psi(-i\xi)}$ gives zero, hence the difference is zero in $\mathcal S'$.
}\end{proof}

\section{Power Series Expansion of the Solution}

The Cauchy problem \eqref{problem} is well posed, as shown in the previous section. However, no information about the asymptotic behavior of the solution was provided. The aim of this section is therefore to derive the power series expansion of the solution $\Phi_t^\beta$, $\beta \in \R^*$.

The following lemma, called the Multiindex Leibniz formula, is used for multiindex order derivatives of functions products (See \cite{horvath}, Chap.2, Sec.5, Prop. 3).
\begin{lemma}\label{leib} Let $f,g\in \mathcal C^\infty(\R^d)$ be real valued functions. Then, for $\alpha \in \mathbb{N}_0^d$, we have
\begin{equation}
\pa^\alpha (fg)
  = \sum_{\gamma \le \alpha} \binom{\alpha}{\gamma}\,
    \pa^\gamma f \, \pa^{\alpha-\gamma} g.
\end{equation}
\end{lemma}

\begin{proof}
The proof is achieved via induction on $|\alpha|$ by applying the product rule. The two terms at each step are combined by Pascal's identity for the multiindex binomial coefficients.
\end{proof}

We now state the generalized Faà di Bruno theorem, which provides the formula for multiindex-order derivatives of composite functions.

\begin{theorem}\label{faadibruno}
Let $h\in \mathcal C^\infty(\R^d)$, $g\in \mathcal C^\infty(\R)$ be real valued functions. Then, for every  $\alpha\in{ \N_0^d\setminus\{0\}}$,
\begin{equation}\label{faadibrunoeq}
\pa^\alpha\bigl[g(h(\phi))\bigr]
=
\sum_{k=1}^{|\alpha|}
g^{(k)}(h(\phi))
\sum_{(j_\mu)\in\mathcal J_{\alpha,k}}
\frac{\alpha!}{\displaystyle\prod_\mu j_\mu!\,(\mu!)^{j_\mu}}
\prod_\mu \bigl(\pa^\mu h(\phi)\bigr)^{j_\mu},
\end{equation}
\end{theorem}

where 
\begin{equation}\label{Jalphak}
\mathcal J_{\alpha,k}
:=
\left\{
(j_\mu)_{\mu\in\mathbb N_0^d\setminus\{0\}}:
j_\mu\in\mathbb N_0,\ 
\sum_\mu j_\mu=k,\ 
\sum_\mu j_\mu\mu=\alpha
\right\}.
\end{equation}
\begin{proof}
This follows from \cite{hernandez} Theorem 1.1 .
\end{proof}

\begin{corollary}\label{faadibrunopol}
If $V:\R^d\to\R$ is a polynomial, $g\in \mathcal C^\infty(\R)$, and $\alpha\in{ \N_0^d\setminus\{0\}}$, then $\pa^\alpha[g(V(\phi))]$ is a finite sum of terms of the form
\begin{equation}
P_{\alpha,k}(\phi)\,g^{(k)}(V(\phi)),\qquad 1\leq k \leq |\alpha|,
\end{equation}
where each $P_{\alpha,k}$ is a polynomial in $\phi$.
\end{corollary}

\begin{proof}
By applying Theorem \ref{faadibruno}, and by the fact that for every multiindex $\mu\leq\alpha$, 
$\pa^\mu V$ is again a polynomial, the result follows.
\end{proof}

The following theorem establishes the differentiability property of the initial condition $f$ with respect to $\beta\in \R^*$ in $\mathcal S(\R^d)$.

\begin{theorem}\label{infdiff}For $m\in\N$, the { map $\beta\mapsto f(\beta V)$} is infinitely differentiable { on $\R^*$} and
    \begin{equation}\label{differen}\pa^{m}_\beta f(\beta V(\phi)):=\frac{\dd^{m}}{\dd \beta^{m}}f(\beta V(\phi))=V^{m}(\phi)f^{(m)}(\beta V(\phi))\qquad \text{ in } \Sspace(\R^d),
\end{equation}
where $V^m(\phi)=\bigl(V(\phi)\bigr)^m$.
\end{theorem}
\begin{proof}
Let $\beta\ne0$ and let $J\subset\R^*$ be a compact interval  containing $\beta$. By the coercivity of $V$, $|V(\phi)|\ge c|\phi|^q$ on $\{|\phi|>r\}$ for some $r,c,q>0$. The chain rule and rapid decay of every derivative of $f$ imply
\begin{equation}
\sup_{\sigma\in J}\big\|V^r f^{(m)}(\sigma V)\big\|_{\alpha,N}<\infty,
\end{equation}
for all $r,N,m\in\N_0$ and $\alpha\in\N_0^d$. Since spatial derivative is a finite sum of polynomial factors times $f^{(m+j)}(\sigma V)$, the lower bound absorbs any polynomial factor uniformly for $\sigma\in J$. By Taylor's integral remainder, we get
\begin{equation}\label{taylor1}
\frac{V^m f^{(m)}((\beta+h)V)-V^m f^{(m)}(\beta V)}{h}
-V^{m+1}f^{(m+1)}(\beta V)
=h\!\int_0^1(1-\theta)V^{m+2}f^{(m+2)}((\beta+\theta h)V)\,d\theta.
\end{equation}

The Schwartz seminorm of the L.H.S is at most $C_{\alpha,N,m,\beta}|h|$ for small $h$. Finally, proceeding by induction on $m$ proves \eqref{differen}.
\end{proof}

\begin{definition}
\label{def:nu-pairing}
For every $f\in\mathcal S(\mathbb R^d)$, we use the notation $\langle f\rangle_{\nu_t}$ to indicate the canonical pairing between $f\in \Sspace$ and $\nu_t\in \Sspace'(\R^d)$, that is
\begin{equation}\label{eq:nu-pairing}
\langle f\rangle_{\nu_t}
:=
\langle \nu_t,f\rangle.
\end{equation}
Under Assumption \ref{assumptions}(3), $\nu_t$ is the probability measure with characteristic function $e^{t\Psi(-i\xi)}$. The analyticity of the Fourier transform of $r$ implies finite polynomial moments of every order; the same holds for the Gaussian and compound-Poisson convolution in \eqref{levysymbol}. We use the same notation for any $\nu_t$-integrable function:
$$
\langle f\rangle_{\nu_t}:=\int_{\R^d}f(\phi)\,\nu_t(d\phi).
$$
\end{definition}

We extend the pairing
\eqref{eq:nu-pairing} { by this integral} to the monomial
$f(\phi)=\phi^\alpha$. Since
\begin{equation}\label{eq:fourier-pairing}
\widehat{\nu_t}(\xi)
=
\left\langle
e^{i\langle\xi,\phi\rangle}
\right\rangle_{\nu_t}
=
e^{t\Psi(-i\xi)},
\end{equation}
{ differentiation under the integral} with respect to $\xi$ gives
\begin{equation}\label{eq:fourier-moment-derivative}
\partial_\xi^\alpha\widehat{\nu_t}(\xi)
=
i^{|\alpha|}
\left\langle
\phi^\alpha e^{i\langle\xi,{ \phi}\rangle}
\right\rangle_{\nu_t}.
\end{equation}

Setting $\xi=0$ in equation \eqref{eq:fourier-moment-derivative}, we obtain
\begin{equation}\label{distmoment}
\langle\phi^\alpha\rangle_{\nu_t}
=
i^{-|\alpha|}
\left.
\partial_\xi^\alpha e^{t\Psi(-i\xi)}
\right|_{\xi=0}.
\end{equation}

This is called the $\alpha$-th
distributional moment of $\nu_t$. { Under Assumption \ref{assumptions}, this is an identity of ordinary moments. For a general tempered distribution, the derivative formula is a formal extension and is not a canonical pairing with a non-Schwartz function.} Since every polynomial is a finite linear combination of monomials, the
distributional moments determine its pairing with $\nu_t$. In
particular, for the polynomial $V$ defined in \eqref{Vdef2}, we set
\begin{equation}\label{polymoment}
\langle V\rangle_{\nu_t}
:=
\sum_{|\alpha|\leq\bar p}
c_\alpha\langle\phi^\alpha\rangle_{\nu_t}{ .}
\end{equation}

Denote $V_\phi(\varphi)=V(\phi-\varphi)$. 
We now have the following theorem

\begin{theorem}\label{extend} For every $m\in\N_0$, the $m^\text{th}$ derivative of the solution $\Phi_t^\beta$ extends to zero, i.e

\begin{equation}
\lim_{\beta\rightarrow0}\pa^m_\beta \Phi^\beta_t=f^{(m)}(0)\langle V^m_\phi\rangle_{\nu_t}\qquad\text{ in }\Sspace'(\R^d),
\end{equation}
\end{theorem}

\medskip
The proof of Theorem \ref{extend} is based on the following lemma.

\begin{lemma}\label{physical}
Under Assumptions \ref{assumptions} and Theorem \ref{fschwartz} 
\begin{equation}\label{physical-limit}
\lim_{\beta\to0} V^m f^{(m)}(\beta V)=f^{(m)}(0)V^m
\qquad\text{ in }\Sspace'(\R^d), 
\end{equation}
\end{lemma}
for every $m\in\N_0$.
\begin{proof}
For each $\psi\in\mathcal S(\R^d)$, the integrand $V^m(\phi)f^{(m)}(\beta V(\phi))\psi(\phi)$ converges pointwise to \\$f^{(m)}(0)V^m(\phi)\psi(\phi)$ and is bounded by $\|f^{(m)}\|_\infty|V^m\psi|\in L^1$. Hence, \eqref{physical-limit} follows by a direct application of the dominated convergence theorem.

\end{proof}

\subsection{Proof of Theorem \ref{extend}}

By Theorem \ref{existuniq}, for $\beta\neq0$
\begin{equation}
\Phi_t^\beta(\phi)=(\nu_t*f_\beta)(\phi), \qquad \phi\in\R^d,
\end{equation}
and by Fourier transform and equation \eqref{nu}
\begin{equation}\label{fourier-sol}
\widehat{\Phi_t^\beta}(\xi)=e^{t\Psi(-i\xi)}\,\widehat{f_\beta}(\xi), \qquad \xi\in\R^d.
\end{equation}

Hence, by Theorem \ref{nutemp}, Theorem \ref{infdiff}, and the continuity of Fourier transform, we differentiate with respect to $\beta$
\begin{equation}\label{dbeta-fourier}
\pa_\beta^m\widehat{\Phi_t^\beta}
=\widehat{\pa_\beta^m\Phi_t^\beta}
=
e^{t\Psi{ (-i\xi)}}\,{ \widehat{V^m f^{(m)}(\beta V)}}.
\end{equation}

By \cite{horvath} (Chap. 4, Sec. 11, Prop. 5) and Lemma \ref{physical}
\begin{equation}
\lim_{\beta\rightarrow0}e^{t\Psi{ (-i\xi)}}\,\widehat{V^m f^{(m)}(\beta V)}
=
f^{(m)}(0)\,e^{t\Psi{ (-i\xi)}}\,\widehat{V^m}
\qquad\text{in }\Sspace'(\R^d),
\end{equation}
that is
\begin{equation}\label{lastconv}
\lim_{\beta\rightarrow0}\widehat{\pa_\beta^m\Phi_t^\beta}
=
f^{(m)}(0)\,e^{t\Psi{ (-i\xi)}}\,\widehat{V^m}
\qquad\text{in }\Sspace'(\R^d).
\end{equation}

By continuity of the inverse Fourier transform operator, we obtain from equation \eqref{lastconv}
\begin{equation}
\lim_{\beta\rightarrow0}\pa_\beta^m\Phi_t^\beta
=
f^{(m)}(0)\,\F^{-1}\!\big(e^{t\Psi{ (-i\xi)}}\,\widehat{V^m}\big)
\qquad\text{in }\Sspace'(\R^d).
\end{equation}

{ The finite polynomial moments of $\nu_t$ give
$$
\mathcal F^{-1}\!\bigl(e^{t\Psi(-i\xi)}\widehat{V^m}\bigr)(\phi)
=(\nu_t*V^m)(\phi)
=\int_{\R^d}V(\phi-\varphi)^m\,\nu_t(d\varphi)
=\langle V_\phi^m\rangle_{\nu_t}.
$$

The convolution is well defined and is a polynomial in $\phi$.}
Hence

\begin{equation}
\lim_{\beta\rightarrow0}\pa^m_\beta \Phi^\beta_t=f^{(m)}(0)\langle V^m_\phi\rangle_{\nu_t}\qquad\text{ in }\Sspace'(\R^d).
\end{equation}
which proves the theorem.\hfill $\qed$

Theorem \ref{extend} extends the solution and its derivatives to zero. The solution therefore admits an asymptotic expansion around zero in the parameter $\beta$.

\begin{theorem}\label{expansion} The formal $V-$expansion of the solution $\Phi_t^\beta$ is given by 

\begin{equation}
    \Phi_t^\beta(\phi){ \mathrel{\sim}}\sum_{m=0}^\infty \beta^m\frac{f^{(m)}(0)}{m!}\langle V^m_\phi(\varphi)\rangle_{\nu_t}, \qquad { \beta\to0}, 
\end{equation}
where $\nu_t$ is the tempered distribution given by equation \eqref{nu}.
The expansion holds for fixed $t\geq0$ and $\phi\in\R^d$, and in $\mathcal S^{\prime}(\R^d)$. 

\end{theorem}
\begin{proof}
    The proof follows from Theorem \ref{infdiff}, Theorem \ref{extend}, and { Taylor's formula with integral remainder. For each $N\in\N_0$,
$$
\left|\Phi_t^\beta(\phi)-\sum_{m=0}^N
\frac{f^{(m)}(0)}{m!}\beta^m\langle V_\phi^m\rangle_{\nu_t}\right|
\leq\frac{|\beta|^{N+1}}{(N+1)!}\|f^{(N+1)}\|_\infty
\int_{\R^d}|V(\phi-\varphi)|^{N+1}\,\nu_t(d\varphi).
$$

The integral is finite and grows at most polynomially in $\phi$. This proves the finite-order asymptotics pointwise and after pairing with a Schwartz function.}
\end{proof} 

\begin{remark}
    By Theorem \ref{expansion}, the solution $\Phi_t^\beta$ admits a formal expansion whose coefficients are distributional moments of the powers of $V_\phi$. These moments can be reorganized combinatorially and represented as sums indexed by a specific class of graphs. Each graph corresponds to one contribution in the expansion, and its value is determined by a fixed evaluation procedure, referred to as the \emph{Feynman rules}.

\end{remark}

\section{Graphical Expansion}

 In this section, we introduce the class of graphs associated with the moments
$\langle V_\phi^m\rangle_{\nu_t}$, define their Feynman rules, and rewrite the formal $V$-expansion of $\Phi_t^\beta$ as a graphical expansion.

Let $m\in\mathbb N$, we denote $[m]:=\{1,\dots,m\}$. For a partition $\pi\in\mathcal P([m])$ of $[m]$, 
\begin{equation}
\pi=\{B_1,\dots,B_{|\pi|}\},
\end{equation}
we denote $|\pi|$ as the number of blocks in $\pi$, and $|B|$ as the cardinality of each block $B\in\pi$, and for $i,j\in [m]$, we write $i\sim_\pi j$ whenever $i$ and $j$ belong to the same block of $\pi$.

Given two partitions $\pi$ and $\sigma$, we can compare them according to whether one is finer or coarser than the other. The following definition formalizes this comparison.
\begin{definition}
Let $\sigma,\pi\in\mathcal P([m])$. 
\begin{enumerate}

\item We say that $\sigma$ is finer than $\pi$ (or that $\pi$ is coarser than $\sigma$) and we denote it $\sigma\le \pi$ whenever every block of $\sigma$
is contained in some block of $\pi$.
Equivalently, $\pi$ can be obtained by repeatedly merging blocks of $\sigma$.

\item We define the join of $\sigma$ and $\pi$, denoted $\sigma\vee\pi$, as the the smallest partition which is coarser than both $\sigma$ and $\pi$.

\item If $\sigma\le \pi$, we define the \textit{Möbius inversion function} $\mu(\sigma,\pi)$
\begin{equation}
\mu(\sigma,\pi)
=
(-1)^{|\sigma|-|\pi|}
\prod_{B\in\pi}
\bigl(k_B-1\bigr)!,
\end{equation}
where $k_B$ is the number of blocks of $\sigma$ contained in $B$.
If $\sigma\not\le\pi$, then $\mu(\sigma,\pi)=0$.
\end{enumerate}
\end{definition}

\begin{proposition}\label{prop:refinement-poset}
Let $m\in\N$. The relation "$\le$" on $\mathcal P([m])$ is a partial order.
\end{proposition}

\begin{proof}
Reflexivity: every block of $\sigma$ is contained in itself, so $\sigma\le\sigma$.
Transitivity: if $\sigma\le\pi$ and $\pi\le\rho$, then every block of $\sigma$ lies in a block of $\pi$,
and that block lies in a block of $\rho$, hence $\sigma\le\rho$.
Antisymmetry: if $\sigma\le\pi$ and $\pi\le\sigma$, then each block of $\sigma$ is contained in a block of $\pi$
and conversely; because the blocks are disjoint and cover $[m]$, this forces $\sigma=\pi$.
\end{proof}

The truncated moments of $\nu_t$ are defined through the logarithm of the Fourier transform by the following definition

\begin{definition}\label{distruncdef}
For $\alpha=(\alpha_1,\cdots,\alpha_d)\in \N_0^d$ { with $|\alpha|\geq1$}, the $\alpha$-th truncated distributional moment of $\nu_t$ is defined by
\begin{equation}\label{distrunc}
\langle \phi^\alpha\rangle_{\nu_t}^{T}:=\langle { \underbrace{\phi_1,\ldots,\phi_1}_{\alpha_1},\ldots,\underbrace{\phi_d,\ldots,\phi_d}_{\alpha_d}}\rangle_{\nu_t}^{T}
=
t\,{ (-1)^{|\alpha|}}\pa_\xi^\alpha\Psi(\xi)\big|_{\xi=0},
\end{equation}
that is
\begin{equation}
\langle \phi^\alpha\rangle_{\nu_t}^{T}=t{ (-1)^{|\alpha|}|\alpha|!}C^{(|\alpha|)}_{{ \underbrace{1\cdots1}_{\alpha_1}\cdots\underbrace{d\cdots d}_{\alpha_d}}}.
\end{equation}
\end{definition}

\begin{theorem}\label{cluster}
Let $\phi=(\phi_1,\dots,\phi_d)\in\R^d$, and let
$i_1,\dots,i_m\in\{1,\dots,d\}$. Then, for every $m\ge1$,
\begin{equation}\label{linked}
\Big\langle \prod_{r=1}^m \phi_{i_r}\Big\rangle_{\nu_t}
=
\sum_{\pi\in\mathcal P([m])}
\prod_{B\in\pi}
\Big\langle \prod_{r\in B}\phi_{i_r}\Big\rangle_{\nu_t}^T.
\end{equation}
and
\begin{equation}\label{reverse-linked}
\Big\langle \prod_{r=1}^m \phi_{i_r}\Big\rangle_{\nu_t}^T
:=
\sum_{\pi\in\mathcal P([m])}\mu(\pi,\hat 1)
\prod_{B\in\pi}
\Big\langle \prod_{r\in B}\phi_{i_r}\Big\rangle_{\nu_t},
\end{equation}
where $\hat{1}$ is the maximal element of $\mathcal P([m])$, i.e, the coarsest partition.
\end{theorem}

\begin{proof}
For the proof, we refer to \cite{schwinger} and \cite{peccati}.
\end{proof}
\begin{remark}
Theorem \ref{cluster} allows a back and forth transition between moments and truncated moments of the tempered distribution $\nu_t$.
\end{remark}

The formal power series given in Theorem \ref{expansion} can be expanded in terms of the powers of $V$. We introduce the index set $\Omega(p_1,\dots,p_m)$ as follows 
\begin{equation}\label{indexset}
\Omega(p_1,\dots,p_m)
:=
\bigcup_{n=1}^{m}\{(1,n),\dots,(p_n,n)\}.
\end{equation}

For $I_l=\{(a_1,b_1),\dots,(a_q,b_q)\}$, we denote 
\begin{equation}\label{Cil}
C_{I_l}:=C^{(q)}_{{ i_{a_1}^{b_1}\cdots i_{a_q}^{b_q}}},
\end{equation}
where $C^{(k)}_{{ i_1\cdots i_k}}$
are the symmetric tensors of the pseudodifferential operator given in equation \eqref{pseudo}.

\begin{theorem}\label{expansion2}
{ The} solution $\Phi_t^\beta$ of the Cauchy problem
\eqref{problem} admits the asymptotic expansion { as $\beta\to0$}
\begin{equation}
\begin{aligned}
\Phi_t^\beta(\phi)
&{ \mathrel{\sim}}
\sum_{m=0}^{\infty}
\frac{f^{(m)}(0)}{m!}\beta^m
\sum_{p_1,\dots,p_m=0}^{\bar p}
\sum_{\substack{
K\subseteq \Omega\\
I\in\mathcal{P}\bigl(\Omega\setminus K\bigr)\\
I=\{I_1,\dots,I_k\}
}}
\sum_{i^{1}_1,\dots,i^{1}_{p_1}=1}^{d}
\cdots
\sum_{i^{m}_1,\dots,i^{m}_{p_m}=1}^{d}
\lambda^{(p_1)}_{i^{1}_1,\dots,i^{1}_{p_1}}
\cdots
\lambda^{{ (p_m)}}_{i^{m}_1,\dots,i^{m}_{p_m}}
\\
&\qquad\qquad\times
\prod_{(s,q)\in K}
\bigl({ +}\phi_{i^{q}_s}\bigr)
\prod_{l=1}^{k} t\,{ |I_l|!}C_{I_l},
\end{aligned}
\end{equation}
where $\Omega:=\Omega(p_1,\dots,p_m)$.
\end{theorem}

\begin{proof}
First, we expand the $m$ labeled copies of $V(\phi-\varphi)$ using the definition of the polynomial $V$ in \eqref{Vdef}. Every  $K\subseteq\Omega$ selects the factors $+\phi_i$ while the remaining factors $-\varphi_i$ are grouped by a partition $I$ using Theorem \ref{cluster}. By \eqref{distrunc},  a block $I_l$ contributes
\begin{equation}
\big\langle-\varphi_{i_a}:a\in I_l\big\rangle_{\nu_t}^T
=(-1)^{|I_l|}t(-1)^{|I_l|}|I_l|!C_{I_l}
=t|I_l|!C_{I_l}.
\end{equation}

Then, summing over $K$, $I$, the degrees $p_q$, and the indices gives the coefficient of $\beta^m$. Finally, we conclude via Theorem \ref{expansion}.
\end{proof}

\begin{definition}
    A (finite) graph is a pair $G=(V,E)$, where $V$ is a set of vertices and 
\begin{equation}
{ E\text{ is a finite family of pairs }\{u,v\},\qquad u,v\in V,\ u\neq v.}\end{equation}

Different edges may have the same endpoints. Denote by $E(G)$ the edge { family} of $G$, and by $V(G)$ the vertex set of $G$.
If no loops or multiple edges are allowed, $G$ is called a \emph{simple} graph.

In such graphs, we call a path, any sequence of distinct vertices $(v_0,\dots,v_k)$ such that 

\begin{equation}
\{v_{i-1},v_i\}\in E\quad  \text{for all } i=1,\dots,k.
\end{equation}
\end{definition}

For more details about graph theory concepts, we refer the reader to \cite{bollobas}, \cite{oystein}, and \cite{serre}.

\begin{definition}\label{Feyndef}
A \emph{generalized amputated Feynman graph} is a graph with three types of vertices:
\emph{inner full}, \emph{inner empty}, and \emph{outer empty} vertices.
Inner full vertices are distinguishable and have distinguishable legs; empty vertices are non-distinguishable and have non-distinguishable legs.
Each outer empty vertex is met by exactly one edge. { Each inner empty vertex is met by at least one edge. Every edge joins an inner full vertex to an empty vertex. Parallel edges are distinguished by the labels of their legs at the full vertex.}

Fix $m\in\mathbb N_0$. Let $p_i$, $1\leq i \leq m$, be the numbers of legs of the $m$ inner full vertices, with ${ 0\leq}p_j\le \bar p$ and $\lambda^{(p_j)}\neq 0$.
The collection of all such graphs is denoted by ${\mathcal F}(m)$.
\end{definition}

Let $\Omega(p_1,\dots,p_m)$ be as in equation \eqref{indexset}. For each $j=1,\dots,m$, we label the legs of the $j$th inner full vertex by
\begin{equation}\label{pj}
(1,j),\dots,(p_j,j),
\end{equation}
so that the set of all legs of inner full vertices is in bijection with $\Omega(p_1,\dots,p_m)$. 

We will now construct a one to one correspondence between the pairs $(K,I)$ and the Feynman amputated graphs $\mathcal F(m)$.
Let $G$ be a generalized amputated Feynman graph with $m$ inner full vertices.
\begin{enumerate}
\item {Define $K\subseteq \Omega(p_1,\dots,p_m)$ as the subset corresponding to the legs of inner full vertices of $G$ { that end at outer empty vertices, }{ with labels given in equation \eqref{pj}.}}
\item Let $k$ be the number of inner empty vertices of $G$. For each $q=1,\dots,k$, let
$I_q\subseteq \Omega(p_1,\dots,p_m)$, $1\leq q \leq k$, be the subset corresponding to those legs of inner full vertices that run into the $q$th inner empty vertex, and set
\begin{equation}
I:=\{I_1,\dots,I_k\}.
\end{equation}
Then $I\in \mathcal P(\Omega(p_1,\dots,p_m)\setminus K)$. 
\end{enumerate}

Conversely, let $K\subseteq \Omega(p_1,\dots,p_m)$ and let
$I\in \mathcal P(\Omega(p_1,\dots,p_m)\setminus K)$ with $I=\{I_1,\dots,I_k\}$ be given.
Construct a graph $G\in {\mathcal F}(m)$ as follows:
\begin{enumerate}
\item Draw $m$ inner full vertices with $p_1,\dots,p_m$ legs, draw $k$ inner empty vertices with $|I_1|,\dots,|I_k|$ legs, and draw $|K|$ outer empty vertices. 
\item For each $q=1,\dots,k$, connect the legs of the $q$th inner empty vertex to the legs of inner full vertices corresponding to the elements of $I_q$. 
\item Connect the legs of the outer empty vertices to the legs of inner full vertices corresponding to the elements of $K$.
\end{enumerate}

\begin{remark} The above construction establishes the one-to-one correspondence between pairs $(K,I)$ and graphs in ${\mathcal F}(m)$, see \cite{main}, Lemma 2.9 for more details. 
\end{remark}
\begin{definition}\label{feynman}
Let $G\in{\mathcal F}(m)$, and let $p_1,\dots,p_m$ denote the numbers of legs attached to the inner full vertices of $G$.
For $\phi\in\R^d$ and $t{ \geq}0$, the value $V[G](t,\phi)\in\R$ is defined according to the following procedure:
\begin{enumerate}
\item To each leg of an inner full vertex, assign an index $i_a\in\{1,\dots,d\}$, where $a\in\Omega(p_1,\dots,p_m)$ labels that leg. 
\item For each outer empty vertex, multiply by the factor $({ +}\phi_{i_b})$, where $b\in\Omega(p_1,\dots,p_m)$ is the label of the leg of the inner full vertex that is directly connected to that outer empty vertex. 
\item For each inner empty vertex with $l$ legs, multiply by the factor
\begin{equation}
    t\,{ l!}C_J \;=\; t\,{ l!}C{ ^{(l)}}_{i_{a_1}\cdots i_{a_l}},
\end{equation}
where $J=\{a_1,\dots,a_l\}\subseteq\Omega(p_1,\dots,p_m)$ is the subset of labels corresponding to the $l$ legs of inner full vertices directly connected to that inner empty vertex. 
\item For the $j$th inner full vertex (with $p_j$ legs), multiply by the factor
\begin{equation}
\lambda^{(p_j)}_{i^j_1\cdots i^j_{p_j}}
\end{equation}
and sum over all assignments of the indices $i^j_1,\dots,i^j_{p_j}\in\{1,\dots,d\}$. 
\end{enumerate}
\end{definition}
\begin{remark}\hfill
\begin{enumerate}

\item  We use the convention that $\mathcal F(0)$ consists of the empty graph only, denoted $\varnothing$, and that
\begin{equation}
V[\varnothing](t,\phi)=1.
\end{equation}

\item By construction, the value of a graph $G$ is obtained by multiplying the local factors associated with { all its vertices} and then summing over all assignments of indices. It is therefore natural to expect this value to factorize when the graph splits into disjoint components.
\end{enumerate}
\end{remark}

\begin{example}
    Consider the third order expansion $G\in \mathcal F(3)$ with the following partition $(K,I)$, $I=\{I_1,I_2\}$ for $p_1=3, p_2=2,p_3=4$: 
    \begin{equation}
        K=\{{ (3,1),(3,3),(4,3)}\}\quad I_1=\{(1,1),(1,2),(2,1)\}\quad I_2=\{(2,2),{ (1,3),(2,3)}\}.
    \end{equation} 

Then the associated graph is given by Figure (\ref{graph1}):
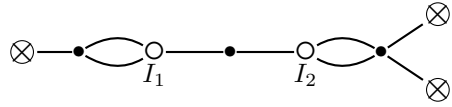
\begin{figure}[h]
    \centering
\begin{tikzpicture}
    \node[node] (p1) at (0,0) {$\bullet$};
    \node[node] (p2) at (-2,0) {$\bullet$};
    \node[node] (p3) at (2,0) {$\bullet$};
    \node[O] (I1) at (-1,0) {};
    \node[O] (I2) at (1,0) {};
    \node (I11) at (-1,-0.3) {$I_1$};
    \node (I22) at (1,-0.3) {$I_2$};

    \node[C] (K1) at (-2.75,0) {};
    \node[C] (K2) at (2.75,0.5) {};
    \node[C] (K3) at (2.75,-0.5) {};

    \draw[thick, shorten <=-0.1pt, shorten >=-0.1pt]
    (K1)--(p2)
    (I1)--(p1)
    (p1)--(I2)
    (I2) to[out=30,in=150] (p3)
    (I2) to[out=-30,in=-150] (p3)
    (p2) to[out=30,in=150] (I1)
    (p2) to[out=-30,in=-150] (I1)
    (p3)--(K2)
    (p3)--(K3);
\end{tikzpicture}
    \caption{Construction of $G\in\mathcal F(3)$.}
    \label{graph1}
\end{figure}
\end{example}

The following proposition provides the factorization property of the Feynman rules.

\begin{proposition}\label{multiplicativity}
Let $m_1,m_2\in\mathbb N_0$, and let $G_1\in \mathcal F(m_1)$ and $G_2\in \mathcal F(m_2)$ be two generalized amputated Feynman graphs. Let 
\begin{equation}
G=G_1\sqcup G_2,
\end{equation}
where $\sqcup$ denotes the disjoint union{ , with the full-vertex labels of $G_2$ shifted by $m_1$}. Then $G\in \mathcal F(m_1+m_2)$ and
\begin{equation}
V[G](t,\phi)=V[G_1](t,\phi)\,V[G_2](t,\phi).
\end{equation}
\end{proposition}

\begin{proof}
Since $G_1\in\mathcal F(m_1)$ and $G_2\in\mathcal F(m_2)$, the graphs $G_1$ and $G_2$ have respectively $m_1$ and $m_2$ inner full vertices. Their disjoint union $G=G_1\sqcup G_2$ has $m_1+m_2$ inner full vertices and still satisfies the defining properties of a generalized amputated Feynman graph given in Definition \ref{Feyndef}, hence $G\in\mathcal F(m_1+m_2)$. By Definition \ref{feynman}, the value $V[G](t,\phi)$ is obtained by multiplying the factors assigned to the vertices of $G$ and summing over all index assignments. 

Since $G=G_1\sqcup G_2$, both the factors and the index assignments split into indices associated with $G_1$ and $G_2$ respectively. Therefore the evaluation defining $V[G](t,\phi)$ factorizes into the product of the evaluations defining $V[G_1](t,\phi)$ and $V[G_2](t,\phi)$. 

Thus
\begin{equation}
V[G](t,\phi)=V[G_1](t,\phi)\,V[G_2](t,\phi).
\end{equation}
\end{proof}

\begin{theorem}\label{graphexpansion}
    The solution $\Phi_t^\beta$ of the Cauchy problem \eqref{problem} has the following graphical expansion { as $\beta\to0$}
\begin{equation}
    \Phi_t^\beta(\phi){ \mathrel{\sim}}\sum_{m=0}^\infty \beta^m\frac{f^{(m)}(0)}{m!}\sum_{G\in \mathcal F(m)} V[G](t,\phi)
\end{equation}
\end{theorem}

\begin{proof} The proof is a straightforward calculation using Theorem \ref{expansion2} and Definition \ref{feynman}.

\end{proof}

\begin{example}
We consider the exponential initial condition ${ f(\beta V)=e^{-\beta V}}$, where
$V(\phi)=\phi^4$, { for $\beta>0$. We use the Schwartz function in \eqref{eq:smoothed-f} with $P\equiv1$; it agrees with $e^{-x}$ for $x\geq0$ and near zero.} In particular, $V$ is an
even quartic potential, also called a $\phi^4$-type potential, which has applications in describing stimulated side scattering processes, see e.g \cite{wkb}. Assume moreover that the pseudo-differential operator
has vanishing odd cumulants, namely
\begin{equation}
C^{(2k+1)}=0,\qquad k\ge0.
\end{equation}

Then the first nontrivial { first-order} expansion { in $\beta$} of $\Phi_t{ ^\beta}$ is given by

\begin{equation}
\begin{aligned}
    \Phi_t^{{ \beta}}(\phi)&=1-{ \beta}\left(\phi^4+{ 12}tC^{(2)}\phi^2+{ 24}tC^{(4)}+{ 12}t^2 (C^{(2)})^2\right)+{ O(\beta^2)}\\
    &={1-\beta\Bigg(}
{\begin{tikzpicture}[baseline=-0.4ex,scale=0.65,every node/.append style={text=black}]
\node[node] (v) at (0,0) {$\bullet$};
\node[C] (a) at (-.55,.45) {}; \node[C] (b) at (-.55,-.45) {};
\node[C] (c) at (.55,.45) {}; \node[C] (d) at (.55,-.45) {};
\draw[thick] (v)--(a) (v)--(b) (v)--(c) (v)--(d);
\end{tikzpicture}}
{+6\,}
{\begin{tikzpicture}[baseline=-0.4ex,scale=0.65,every node/.append style={text=black}]
\node[node] (v) at (0,0) {$\bullet$};
\node[C] (a) at (-.6,.45) {}; \node[C] (b) at (-.6,-.45) {};
\node[O] (c) at (.75,0) {};
\draw[thick] (v)--(a) (v)--(b) (v) to[out=40,in=140] (c) (v) to[out=-40,in=-140] (c);
\end{tikzpicture}}
{+}
{\begin{tikzpicture}[baseline=-0.4ex,scale=0.65,every node/.append style={text=black}]
\node[node] (v) at (0,0) {$\bullet$}; \node[O] (c) at (1,0) {};
\draw[thick] (v) to[out=65,in=115] (c) (v) to[out=22,in=158] (c)
(v) to[out=-22,in=-158] (c) (v) to[out=-65,in=-115] (c);
\end{tikzpicture}}
{+3\,}
{\begin{tikzpicture}[baseline=-0.4ex,scale=0.65,every node/.append style={text=black}]
\node[node] (v) at (0,0) {$\bullet$};
\node[O] (a) at (-.75,0) {}; \node[O] (b) at (.75,0) {};
\draw[thick] (v) to[out=140,in=40] (a) (v) to[out=-140,in=-40] (a)
(v) to[out=40,in=140] (b) (v) to[out=-40,in=-140] (b);
\end{tikzpicture}}
{\Bigg)+O(\beta^2),\qquad\beta\to0.}
\end{aligned}
\end{equation}

\end{example}
\section{Logarithm of the Solution} 

The study of the { formal} logarithm of the solution $\Phi_t^\beta$ { at $\beta=0$}, leads naturally to a graph expansion involving { contributions connected after grouping full vertices according to a partition}, reflecting the fact that the logarithm extracts the cumulant structure, or equivalently the truncated-moment structure, of the perturbative series, see \cite{beisert,Salmhofer}. More explicitly, the passage from $\Phi_t^\beta$ to its logarithm is equivalent to passing from a general graph expansion to its { counterpart connected through these blocks}. To formulate this precisely, we first introduce the appropriate notion of connectedness for the class of graphs under consideration, and then derive the corresponding expansion for $\log \Phi_t^\beta(\phi)$.

Let $n\in\N$ and denote $J_1=\cdots=J_n:=V_\phi$.
For $B\subset[n]=\{1,\dots,n\}$, we define
\begin{equation}\label{JB}
J_B:=\prod_{i\in B}J_i.
\end{equation}

Clearly, $J_B=V_\phi^{|B|}$. The notation $J_B$ is employed to retain
the labels of the factors belonging to the block $B$. This
allows partitions of $[n]$ to index the block-moment and graph
decompositions.

The following lemma provides the combinatorial step needed to rewrite the multiindex sum in partition notation.

\begin{lemma}\label{lem:multi-to-part}
Let $n\ge 1$ and $1\le k\le n$. Fix integers
\begin{equation}
r_1,\dots,r_k\ge 1,
\qquad
r_1+\cdots+r_k=n.
\end{equation}

Then there is a one-to-one correspondence between:
\begin{enumerate}
\item A decomposition of the $n$ labeled copies of $V_\phi$ into $k$ ordered groups of respective sizes $r_1,\dots,r_k$.

\item An ordered set partition $(B_1,\dots,B_k)$ of $[n]$, that is, a family of nonempty pairwise disjoint subsets $B_1,\dots,B_k\subset [n]$ such that
\begin{equation}
\bigcup_{a=1}^k B_a=[n],
\qquad
|B_a|=r_a
\quad\text{for all }a=1,\dots,k.
\end{equation}
\end{enumerate}

This is equivalent to a representation by pair $(\pi,l)$, where $\pi\in\mathcal P([n])$ and $l:\pi\to\{1,\dots,k\}$
is a bijection labeling the blocks of $\pi$, with the property that for each
$a\in\{1,\dots,k\}$, the unique block $B\in\pi$ satisfying $l(B)=a$ has
cardinality $|B|=r_a$.
\end{lemma}
\begin{proof}

The proof is based on the multinomial correspondence theorem. We refer the reader to (\cite{Comtet}, Theorem A).
\end{proof}

\begin{theorem}\label{logarithm}
Let $\Phi_t^{\beta}$, $\beta>0$ be the solution of the Cauchy problem \eqref{problem}. Assume that $f(0)=1$. Then, as a formal power series in $\beta$,
\begin{equation}
\log\Phi_t^\beta(\phi)=\sum_{n\ge1}\gamma_n\,\beta^n,
\end{equation}
with
\begin{equation}
\gamma_n
=
\sum_{k=1}^{n}\frac{1}{k!}
\sum_{\substack{r_1,\dots,r_k\ge1\\ r_1+\cdots+r_k=n}}
\Big(\prod_{a=1}^k\frac{\kappa_{r_a}}{r_a!}\Big)\,
\big\langle V_\phi^{\,r_1}{ ,\ldots,} V_\phi^{\,r_k}\big\rangle_{\nu_t}^T.
\end{equation}
\end{theorem}

First, we give a sense to the distributional moments of formal power series.
\begin{definition}\label{formalmoment}Let $\mathcal P[\phi]$ be the space of polynomials in $\phi$ for which the
distributional moments with respect to $\nu_t$ are defined (as in equation \eqref{polymoment}). For a formal series, 
\begin{equation}
F(\beta,\phi)=\sum_{n\ge0}\beta^n P_n(\phi),
\qquad P_n\in\mathcal P[\phi],
\end{equation}
we define its distributional moment with respect to $\nu_t$ as follows
\begin{equation}\label{eq:coefficientwise-pairing}
\big\langle F(\beta,\phi)\big\rangle_{\nu_t}
:=
\sum_{n\ge0}\beta^n
\big\langle P_n(\phi)\big\rangle_{\nu_t}.
\end{equation}

The distributional moment of $F(\beta,\phi)$ with respect to $\nu_t$ is understood therefore as a map from the ring of formal power series with coefficients in $\mathcal P[\phi]$ to the ring of formal power series with complex coefficients.
\end{definition}

Now we have the following lemma.
\begin{lemma}\label{exp-linked}
Let $P$ be a polynomial in $\phi\in\R^d$. Assume that the distributional moments
$\langle P(\phi)^n\rangle_{\nu_t}$ are defined for every $n\ge0$ as in equation \eqref{polymoment}. Then, as a formal power series in $s$,
\begin{equation}\label{explinked}
\big\langle e^{sP(\phi)}\big\rangle_{\nu_t}
:=
\sum_{n\ge0}\frac{s^n}{n!}\big\langle P(\phi)^n\big\rangle_{\nu_t}
=
\exp\Big(
\sum_{n\ge1}\frac{s^n}{n!}\big\langle P(\phi)^n\big\rangle_{\nu_t}^T
\Big).
\end{equation}

Equivalently,
\begin{equation}\label{loglinked}
\log\big\langle e^{sP(\phi)}\big\rangle_{\nu_t}
=
\sum_{n\ge1}\frac{s^n}{n!}\big\langle P(\phi)^n\big\rangle_{\nu_t}^T.
\end{equation}
\end{lemma}
\begin{proof}
We expand the right-hand side of \eqref{explinked} as a formal power series. At order $s^n/n!$, its exponential formula sums the products $\prod_{B\in\pi}\langle P^{|B|}\rangle_{\nu_t}^T$, where $\pi\in\mathcal P([n])$. Theorem \ref{cluster} identifies this sum with $\langle P^n\rangle_{\nu_t}$. Taking their formal logarithms yields equation \eqref{loglinked}.

\end{proof}

\subsection{Proof of Theorem \ref{logarithm}}    
Let
\begin{equation}\label{def-K}
K(z):=\log f(z)=\sum_{r\ge1}\frac{\kappa_r}{r!}\,z^r,
\qquad \kappa_r:=(\log f)^{(r)}(0),
\end{equation}
be the formal logarithm of { the Taylor series of} the initial condition $f$. Since $f(0)=1$, the formal power series $\log f$ is well defined, and therefore
\begin{equation}
f(z)=\exp(K(z)),
\end{equation}
is well defined as a formal power series.

Replacing $f$ by $\exp(K)$ in { the formal expansion associated with \eqref{phisol}}, we obtain
\begin{equation}
\Phi_t^\beta(\phi)=\big\langle f(\beta V_\phi)\big\rangle_{\nu_t}
=
\Big\langle \exp\!\big(K(\beta V_\phi)\big)\Big\rangle_{\nu_t},
\end{equation}

Denote $X:=K(\beta V_\phi).$ The constant term of $X$ in $\beta$ vanishes. Thus, in the argument of Lemma~\ref{exp-linked}, the coefficient of $\beta^n$ receives contributions only from $k\leq n$. Applying that argument coefficient by coefficient, with the joint cumulants extended by multilinearity, we get

\begin{equation}\label{logPhi-double}
\begin{aligned}
\log\Phi_t^\beta(\phi)
&=
\log\big\langle e^{X}\big\rangle_{\nu_t}=\sum_{k\ge1}\frac{1}{k!}\,\big\langle X^k\big\rangle_{\nu_t}^T\\
&=\sum_{k\ge1}\frac{1}{k!}
\sum_{r_1,\dots,r_k\ge1}
\Big(\prod_{a=1}^k\frac{\kappa_{r_a}}{r_a!}\Big)\,
\beta^{r_1+\cdots+r_k}\,
\big\langle V_\phi^{\,r_1},\cdots ,V_\phi^{\,r_k}\big\rangle_{\nu_t}^T.
\end{aligned}
\end{equation}

Finally, expressing the logarithm as 
\begin{equation}
\log \Phi_t^\beta(\phi)=\sum_{n\ge1}\gamma_n\,\beta^n,
\end{equation}
and identifying coefficients of $\beta^n$ in \eqref{logPhi-double}, we obtain
\begin{equation}\label{gamma_n}
\gamma_n
=
\sum_{k=1}^{n}\frac{1}{k!}
\sum_{\substack{r_1,\dots,r_k\ge1\\ r_1+\cdots+r_k=n}}
\Big(\prod_{a=1}^k\frac{\kappa_{r_a}}{r_a!}\Big)\,
\big\langle V_\phi^{\,r_1},\cdots, V_\phi^{\,r_k}\big\rangle_{\nu_t}^T.
\end{equation}\hfill $\qed$

\subsection{Graphical Representation of The Logarithm}

In Theorem \ref{logarithm} we proved { a formal} expansion of the logarithm of the solution $\Phi_t^\beta$. In this subsection we will establish an analogue of the graphical representation.

For block variables $J_B$, the notation
$\big\langle \prod_{B\in\pi} J_B \big\rangle_{\nu_t}^{T}$
denotes the truncated moment of the family $(J_B)_{B\in\pi}$, that is

\begin{equation}
\Big\langle \prod_{B\in\pi} J_B \Big\rangle_{\nu_t}^{T}
:=
\big\langle J_B : B\in\pi \big\rangle_{\nu_t}^{T},
\end{equation}
where $J_B$ are defined as in equation \eqref{JB}

We give a sense now to the connectedness of a graph. We start by the labelling of a graph. Let $p_i, p_j$ be two inner full vertices of the graph $G$. We define the relation "is connected" and we denote it by $\sim_G$ between two vertices as follows: $p_i \sim_G p_j$ if and only if there exists a path between these two vertices. Indeed this is an equivalence relation.

In terms of partitions, connectedness can be understood as follows:

\begin{definition}
    A graph $G\in \mathcal F(n)$ is said to be $\sigma$-connected (connected with label $\sigma\in \mathcal P([n])$), { with $\sigma(G)=\sigma$,} if its inner full vertices $(p_i)_{i\in[n]}$ are only connected block-wise, that is, for each block { $B_i\in \sigma$}, $p_k\sim_G p_l$ for $k,l\in B_i$. In other words, $(B_i)_{i\in \{1,\cdots,|\sigma|\}}$ are the equivalence classes of the relation $\sim_G$.
\end{definition} 

\begin{example} The following graph $G\in\F(2)$ is a $\hat{1}$-connected graph corresponding to the partition $\sigma=\{\{1,2\}\}$:
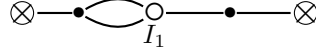
\begin{figure}[h!]
    \centering
\begin{tikzpicture}
    \node[node] (p1) at (0,0) {$\bullet$};
    \node[node] (p2) at (-2,0) {$\bullet$};
    \node[O] (I1) at (-1,0) {};
    \node[C] (I2) at (1,0) {};
    \node (I11) at (-1,-0.3) {$I_1$};

    \node[C] (K1) at (-2.75,0) {};

    \draw[thick, shorten <=-0.1pt, shorten >=-0.1pt]
    (K1)--(p2)
    (I1)--(p1)
    (p1)--(I2)
    (p2) to[out=30,in=150] (I1)
    (p2) to[out=-30,in=-150] (I1);
\end{tikzpicture}
    \caption{The graph $G$}
    \label{graph2}
\end{figure}
\end{example}

The following proposition links the truncated moments with the generalized amputated Feynman graphs in Definition \ref{Feyndef}. 

\begin{proposition}\label{truncated-to-graph}
For every nonempty subset $C\subset [n]$
\begin{equation}\label{trunctograph}
\Big\langle \prod_{i\in C} J_i \Big\rangle_{\nu_t}^{T}
=
\sum_{\substack{G\in \mathcal F(|C|)\\ \sigma(G)=\hat 1_C}} V[G](t,\phi),\qquad t\in\R_+\quad, \phi\in\R^d.
\end{equation}
In other words, the truncated moment over the block $C$ is given by the sum of the values of all graphs on $C$ whose connectedness partition is the one-block partition $\hat 1_C$. { Here the full vertices are relabeled by $C$; similarly, $\mathcal F(D)$ below denotes the same graph class with full vertices labeled by $D$.}
\end{proposition}

\begin{proof}{
We proceed by induction on $|C|$. Equation \eqref{trunctograph} is valid when $|C|=1$. Assume now it is valid for some $n=|C|$. For a nontrivial partition $\tau\in\mathcal P(C)$, we apply the induction hypothesis on each block and then Proposition \ref{multiplicativity}. The product of the block cumulants equals the sum of $V[G]$ over graphs with connectedness partition $\sigma(G)=\tau$. On the other hand, the moment expansion in Theorem \ref{expansion2} together with Definition \ref{feynman}, gives the full moment as a sum over all $G\in\mathcal F(C)$. Finally, Theorem \ref{cluster} gives the same moment as a sum over $\tau\in\mathcal P(C)$ of products of block cumulants. We conclude by cancelling contributions of $\tau\ne\hat1_C$, which proves the Proposition.
}\end{proof}

We have now the following graphical representation of the logarithm of the solution.

\begin{theorem}\label{logexpgraphthm} { Assume that $f(0)=1$.} The { formal} logarithm of the solution $\Phi_t^\beta$ admits the following graphical expansion { in $\beta$} 
\begin{equation}\label{logexpgrapheq}
\log\Phi_t^\beta(\phi)
=
\sum_{n\ge1}\beta^n
\sum_{\pi\in\mathcal P([n])}
w(\pi)\;
\sum_{\substack{G\in\mathcal F(n)\\ \sigma(G)\vee\pi=\hat 1}} V[G](t,\phi),
\end{equation}
with 
\begin{equation}
w(\pi):=
\frac{1}{n!}\,
\prod_{B\in\pi}\kappa_{|B|}. 
\end{equation}
\end{theorem}

\begin{proof} By Theorem \ref{logarithm}, it suffices to prove that
\begin{equation}\label{eq:gamma_n-partition}
\gamma_n
=
\sum_{\pi\in\mathcal P([n])}
w(\pi)\;
\Big\langle \prod_{B\in\pi} J_B\Big\rangle_{\nu_t}^T.
\end{equation}

Starting from equation (\ref{gamma_n})
\begin{align}
\gamma_n
&=\sum_{k=1}^{n}\frac{1}{k!}
\sum_{\substack{r_1,\dots,r_k\ge1\\ r_1+\cdots+r_k=n}}
\Big(\prod_{a=1}^k\frac{\kappa_{r_a}}{r_a!}\Big)\,
\big\langle V_\phi^{\,r_1}{ ,\ldots,} V_\phi^{\,r_k}\big\rangle_{\nu_t}^T
\nonumber\\
&=\sum_{\substack{(m_j)_{j\ge1}\\ \sum_{j\ge1}j\,m_j=n}}
\Big(\frac{n!}{\prod_{j\ge1}(j!)^{m_j}\,m_j!}\Big)\,
\frac{1}{n!}\Big(\prod_{j\ge1}\kappa_j^{m_j}\Big)\,
\big\langle \underbrace{V_\phi^{1},\dots,V_\phi^{1}}_{m_1},
\underbrace{V_\phi^{2},\dots,V_\phi^{2}}_{m_2},\dots\big\rangle_{\nu_t}^T
\nonumber\\
&=\sum_{\pi\in\mathcal P([n])}\frac{1}{n!}\Big(\prod_{B\in\pi}\kappa_{|B|}\Big)\,
\Big\langle \prod_{B\in\pi}J_B\Big\rangle_{\nu_t}^T.
\end{align}

{ To pass from block cumulants to graphs, apply \eqref{reverse-linked} to the family $(J_B)_{B\in\pi}$. Partitions of this family correspond to partitions $\tau\geq\pi$ of $[n]$. By the moment expansion and Proposition \ref{multiplicativity}, a graph $G$ occurs precisely when $\sigma(G)\leq\tau$. Its total coefficient is therefore
$$
\sum_{\tau\geq\pi\vee\sigma(G)}\mu(\tau,\hat1)
=\begin{cases}1,&\pi\vee\sigma(G)=\hat1,\\0,&\pi\vee\sigma(G)\neq\hat1,\end{cases}
$$
by M\"obius inversion on the partition lattice. Consequently, we obtain}
\begin{equation}
\gamma_n
=
\sum_{\pi\in\mathcal P([n])}
w(\pi)\,
\sum_{\substack{G\in\mathcal F(n)\\ \sigma(G)\vee\pi=\hat 1}} V[G](t,\phi).
\end{equation}

Therefore
\begin{equation}
\log\Phi_t^\beta(\phi)
=
\sum_{n\ge1}\beta^n
\sum_{\pi\in\mathcal P([n])}
w(\pi)\,
\sum_{\substack{G\in\mathcal F(n)\\ \sigma(G)\vee\pi=\hat 1}} V[G](t,\phi).
\end{equation}

This concludes the proof.
\end{proof}

\begin{remark} The result given in Theorem \ref{logexpgraphthm} allows us to characterise which graphs $G\in \mathcal F(m)$ contribute to the graphical representation of the logarithm of the solution. It also highlights an additional characterization through the initial condition's logarithm and its nonzero derivatives. 
\end{remark}

 We notice that the representation provided in Theorem \ref{logexpgraphthm} depends on the Feynman amputated graphs given in Definition \ref{feynman}. For this, we introduce a new graphical representation layer added the same graphical representation given in Definition (\ref{feynman}) which encodes the partition $\pi$: 

 \begin{definition}\label{feynman2} Let $\pi\in \mathcal P({ [n]})$ and $G\in \mathcal F(n), n\in\N$, be fixed such that $\pi\vee \sigma{ (G)}=\hat{1}$, and let $V[G](t,\phi)$ be defined as in Definition \ref{feynman}. Then 

 \begin{itemize}
     \item For each block $B_j\in\pi,\, j\in\{1,\cdots,|\pi|\}$, connect the inner full vertices $p_i$ { with $i\in B_j$} with an inner triangular vertex, denoted $\Delta_j$
     \item For each inner triangular vertex $\Delta_j$, assign a coefficient $\kappa_{|B_j|}=(\log{f})^{(|B_j|)}(0)$.
     \item Multiply all coefficients of inner triangular vertices $\Delta$ with $V[G](t,\phi)$.
 \end{itemize}
     
 \end{definition}

\begin{remark} From Definition \ref{feynman2}, we can have some disconnected graphs (i.e $\sigma$-connected with $\sigma<\hat{1}$) that contribute to the logarithmic expansion, due to the fact that they become connected via the condition $\pi\vee\sigma=\hat{1}$. 
\end{remark}
\begin{example} Consider a graph $G\in \mathcal F(3)$ such that $\sigma(G)=\{\{1,2\},\{3\}\}$. For $\pi=\{\{1\},\{2,3\}\}$, we have
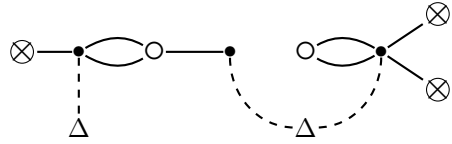
\begin{figure}[h]
    \centering
\begin{tikzpicture}
    \node[node] (p1) at (0,0) {$\bullet$};
    \node[node] (p2) at (-2,0) {$\bullet$};
    \node[node] (p3) at (2,0) {$\bullet$};
    \node[node] (D1) at (-2,-1) {$\Delta$};
    \node[node] (D2) at (1,-1) {$\Delta$};
    \node[O] (I1) at (-1,0) {};
    \node[O] (I2) at (1,0) {};

    \node[C] (K1) at (-2.75,0) {};
    \node[C] (K2) at (2.75,0.5) {};
    \node[C] (K3) at (2.75,-0.5) {};

    \draw[thick, shorten <=-0.1pt, shorten >=-0.1pt]
    (K1)--(p2)
    (I1)--(p1)
    (I2) to[out=30,in=150] (p3)
    (I2) to[out=-30,in=-150] (p3)
    (p2) to[out=30,in=150] (I1)
    (p2) to[out=-30,in=-150] (I1)
    (p3)--(K2)
    (p3)--(K3);
    \draw[thick, dashed, shorten <=-0.1pt, shorten >=-0.1pt]
    (p2)--(D1)
    (p1) to[out=-90,in=180] (D2)
    (p3) to[out=-90,in=0] (D2);
\end{tikzpicture}
    \caption{From $\sigma$-connected to $\hat{1}$-connected}
    \label{graph3}
\end{figure}

In this example, the new graph contributes to the logarithm expansion since the condition $\pi\vee\sigma=\hat{1}$ holds, that is graphically speaking, the graph becomes connected after connecting the inner full vertices corresponding to the $p_i$ via the partition $\pi$. 

To obtain the analytical value of each triangular inner vertex, observe that we have two triangular vertices, $\Delta_1$ corresponding to the block $B_1=\{1\}$ and $\Delta_2$ corresponding to $B_2=\{2,3\}$. Thus $|B_1|=1$ and $|B_2|=2$ and 
\begin{equation}
    \kappa_{|B_1|}=(\log f)'(0),\qquad \kappa_{|B_2|}=(\log f)''(0),\qquad \text{from which} \qquad \tilde{w}(\pi)=(\log f)''(0)(\log f)'(0).
\end{equation}

\end{example}

\section{Borel Summability}

In general, the expansion given in Theorem { \ref{expansion}}
{ does not converge as a power series.} However, it may still encode the
desired analytic function through Borel summation, once analyticity in
a suitable complex domain and a factorial bound on the corresponding
remainder have been established. For the exponential initial condition
$e^{-\beta V}$, $\beta>0$, this methodology was used in \cite{main},
based on a version of Watson's theorem due to Sokal \cite{sokal}.

For $R>0$, we define
\begin{equation}\label{eq:CR}
\mathcal C_R
:=
\left\{
\zeta\in\mathbb C:
\operatorname{Re}(\zeta^{-1})>R^{-1}
\right\}.
\end{equation}

Equivalently, $\mathcal C_R$ is the open disk of radius $R/2$
centered at $R/2$. We also set
\begin{equation}\label{eq:Cinfty}
\mathcal C_\infty
:=
\bigcup_{R>0}\mathcal C_R
=
\left\{
\zeta\in\mathbb C:
\operatorname{Re}\zeta>0
\right\}.
\end{equation}

We first recall the Watson-Sokal theorem.

\begin{theorem}\label{Watson-Sokal}
Let $F$ be analytic in $\mathcal C_R$. Assume that there exist
constants $A,\sigma>0$ and coefficients $(a_m)_{m\geq0}$ such that
\begin{equation}\label{eq:Sokal-remainder}
\left|
F(\zeta)-\sum_{m=0}^{N-1}a_m\zeta^m
\right|
\leq
A\sigma^N N!|\zeta|^N,
\qquad
N\geq1,\quad \zeta\in\mathcal C_R.
\end{equation}

Then the Borel transform
\begin{equation}\label{eq:Borel-transform}
\mathcal BF(\tau)
:=
\sum_{m\geq0}\frac{a_m}{m!}\tau^m
\end{equation}
converges for $|\tau|<\sigma^{-1}$ and admits an analytic continuation
to the strip-like region
\begin{equation}\label{eq:Ssigma}
\mathcal S_\sigma
:=
\left\{
\tau\in\mathbb C:
\operatorname{dist}(\tau,\mathbb R_+)<\sigma^{-1}
\right\}.
\end{equation}

For every $\sigma'>\sigma$, there exists $K_{\sigma'}>0$ such that
\begin{equation}\label{eq:Borel-growth}
|\mathcal BF(\tau)|
\leq
K_{\sigma'}
\exp\left(\frac{|\tau|}{R}\right),
\qquad
\tau\in\mathcal S_{\sigma'}.
\end{equation}

Moreover, for every $\zeta\in\mathcal C_R$,
\begin{equation}\label{eq:Borel-Laplace}
F(\zeta)
=
\frac{1}{\zeta}
\int_0^\infty
e^{-\tau/\zeta}\mathcal BF(\tau)\,d\tau.
\end{equation}

Conversely, suppose that $B$ is analytic in some
$\mathcal S_\sigma$ and, for every $\sigma'>\sigma$, satisfies
\begin{equation}\label{eq:Borel-growth-converse}
|B(\tau)|
\leq
K_{\sigma'}
\exp\left(\frac{|\tau|}{R}\right),
\qquad
\tau\in\mathcal S_{\sigma'}.
\end{equation}

Then
\begin{equation}\label{eq:Laplace-converse}
F(\zeta)
:=
\frac{1}{\zeta}
\int_0^\infty e^{-\tau/\zeta}B(\tau)\,d\tau
\end{equation}
is analytic in $\mathcal C_R$ and satisfies
\eqref{eq:Sokal-remainder}, uniformly on every
$\mathcal C_{R'}$ with $0<R'<R$, with
$$
a_m=B^{(m)}(0).
$$
\end{theorem}

We now impose the assumptions needed for the Borel summability result.

\begin{assumption}\label{Borel}
Fix $t>0$ and $\phi\in\mathbb R^d$. We assume that:
\begin{enumerate}
    \item $V$ is the positive-definite quadratic form
    \begin{equation}\label{eq:quadratic-V}
    V(\varphi)
    =
    \left\langle
    \lambda^{(2)},\varphi^{\otimes2}
    \right\rangle_2,
    \qquad \varphi\in\mathbb R^d;
    \end{equation}

    \item $P$ is a real polynomial of degree $k$, written as
    \begin{equation}\label{eq:P-polynomial}
    P(x)=\sum_{j=0}^k p_jx^j;
    \end{equation}

    \item The jump distribution $r$ in the Lévy symbol is
    compactly supported.
\end{enumerate}
\end{assumption}

We would like to introduce a general initial condition for this part. Let $\chi$ be a smooth function defined by
$\chi\in C^\infty(\mathbb R)$ such that
$$
\chi(x)=0\quad\text{for }x\leq-2,
\qquad
\chi(x)=1\quad\text{for }x\geq-1,
$$
and define for $x\in\R$ 
\begin{equation}\label{eq:smoothed-f}
f(x):=\chi(x)P(x)e^{-x}.
\end{equation}

Then $f\in\mathcal S(\mathbb R)$. Therefore, our initial data is of the form

\begin{equation}\label{newinitcond}
f\bigl(\beta V(\varphi)\bigr)
=
P\bigl(\beta V(\varphi)\bigr)e^{-\beta V(\varphi)},\quad \phi\in\R\quad \beta>0.
\end{equation}

For $\zeta\in\mathcal C_\infty$, we define
\begin{equation}\label{H}
M(\zeta)
:=
\left\langle
e^{-\zeta V_\phi}
\right\rangle_{\nu_t},
\end{equation}
and
\begin{equation}\label{M}
H(\zeta)
:=
\left\langle
P(\zeta V_\phi)e^{-\zeta V_\phi}
\right\rangle_{\nu_t}.
\end{equation}

Under Assumption \ref{Borel}, the Borel summability of $M$ was
proved in \cite[Theorem~2.21]{main}. More precisely, $M$ admits the
asymptotic expansion
\begin{equation}\label{eq:M-expansion}
M(\zeta)
\sim
\sum_{n\geq0}a_n\zeta^n,
\qquad
a_n
=
\frac{(-1)^n}{n!}
\left\langle V_\phi^n\right\rangle_{\nu_t}.
\end{equation}

Its Borel transform is
\begin{equation}\label{eq:BM}
\mathcal BM(\tau)
=
\sum_{n\geq0}\frac{a_n}{n!}\tau^n
=
\sum_{n\geq0}
\frac{(-\tau)^n}{(n!)^2}
\left\langle V_\phi^n\right\rangle_{\nu_t}.
\end{equation}

It admits an analytic continuation to a strip-like neighbourhood of
$\mathbb R_+$ and satisfies the exponential-growth estimates required
by Theorem \ref{Watson-Sokal}.

We now extend this result to our initial condition \eqref{newinitcond}.

\begin{theorem}\label{polyexpborel}
Under Assumption \ref{Borel}, the function $H$ is analytic in
$\mathcal C_\infty$ and satisfies
\begin{equation}\label{eq:H-from-M}
H(\zeta)
=
\sum_{j=0}^k
p_j(-1)^j\zeta^jM^{(j)}(\zeta).
\end{equation}

Moreover, its asymptotic expansion
\begin{equation}\label{eq:H-expansion}
H(\zeta)\sim\sum_{n\geq0}h_n\zeta^n
\end{equation}
is Borel summable, where
\begin{equation}\label{eq:H-coefficients}
h_n
=
\frac{f^{(n)}(0)}{n!}
\left\langle V_\phi^n\right\rangle_{\nu_t}.
\end{equation}

The corresponding Borel transform is
\begin{equation}\label{eq:BH}
\mathcal BH(\tau)
=
\sum_{j=0}^k
p_j(-1)^j\tau^j
\bigl(\mathcal BM\bigr)^{(j)}(\tau).
\end{equation}

Finally,
\begin{equation}\label{eq:H-Borel-Laplace}
H(\zeta)
=
\frac{1}{\zeta}
\int_0^\infty
e^{-\tau/\zeta}
\mathcal BH(\tau)\,d\tau,
\qquad
\zeta\in\mathcal C_\infty.
\end{equation}
\end{theorem}

\begin{proof}{
We differentiate the integral in \eqref{M} with respect to
$\zeta$. This gives
\begin{equation}
M^{(j)}(\zeta)
=
(-1)^j\langle V_\phi^j e^{-\zeta V_\phi}\rangle_{\nu_t},
\end{equation}
and hence \eqref{eq:H-from-M} and the analyticity of $H$.
Using \eqref{eq:M-expansion}, the formal coefficient of $\zeta^n$
in \eqref{eq:H-from-M} is
\begin{equation}
a_n\sum_{j=0}^{\min\{k,n\}}
p_j(-1)^j\frac{n!}{(n-j)!}
=
\langle V_\phi^n\rangle_{\nu_t}
\sum_{j=0}^{\min\{k,n\}}
\frac{p_j(-1)^{n-j}}{(n-j)!}
=
\frac{f^{(n)}(0)}{n!}
\langle V_\phi^n\rangle_{\nu_t}.
\end{equation}

This gives \eqref{eq:H-coefficients} at the level of formal
coefficients. Their Borel transform satisfies \eqref{eq:BH}. It remains to verify the analytic continuation and growth properties
of $\mathcal BH$. By Assumptions \ref{Borel}, compact support of the jump distribution gives
finite exponential moments of linear growth for $\nu_t$. Since $\sqrt{V_\phi}$ has at most
linear growth, we can choose $c>0$ such that
\begin{equation}
A:=\langle e^{c\sqrt{V_\phi}}\rangle_{\nu_t}<\infty.
\end{equation}

It follows for every $N\geq1$,
\begin{equation}
\langle V_\phi^N\rangle_{\nu_t}
\leq Ac^{-2N}(2N)!
\leq A\sigma^N(N!)^2,\quad \sigma=4c^{-2}.
\end{equation}

Taylor's remainder for $e^{-\zeta V_\phi}$ now gives
\begin{equation}
\left|M(\zeta)-\sum_{n=0}^{N-1}a_n\zeta^n\right|
\leq A\sigma^N N!|\zeta|^N,
\qquad \zeta\in\mathcal C_\infty.
\end{equation}

The constants $A$ and $\sigma$ do not depend on $R$. Thus, for
every $\widetilde R>0$ and $\widetilde\sigma>\sigma$,
Theorem \ref{Watson-Sokal} gives
\begin{equation}
|\mathcal BM(w)|
\leq K_{\widetilde\sigma,\widetilde R}
e^{|w|/\widetilde R},
\qquad w\in\mathcal S_{\widetilde\sigma}.
\end{equation}

Now, let $R>0$ be fixed. We choose $\widetilde R>R$, and let
$\sigma<\widetilde\sigma<\sigma'$. Cauchy's estimates bound the
derivatives of $\mathcal BM$ on $\mathcal S_{\sigma'}$ using the
preceding bound on $\mathcal S_{\widetilde\sigma}$. The sum in
\eqref{eq:BH} is finite, and its polynomial factors $\tau^j$ can
be absorbed into $e^{|\tau|/R}$. Therefore $\mathcal BH$ is
analytic in $\mathcal S_\sigma$ and satisfies
\begin{equation}\label{bound}
|\mathcal BH(\tau)|
\leq C_{\sigma',R}e^{|\tau|/R},
\qquad \tau\in\mathcal S_{\sigma'}.
\end{equation}

For $\beta\in(0,R)$, we use the Borel representation
\begin{equation}
M(\beta)=\int_0^\infty e^{-s}\mathcal BM(\beta s)\,ds.
\end{equation}

The bounds in \eqref{bound} allow differentiation under
the integral. Equation \eqref{eq:H-from-M} then gives
\begin{equation}
H(\beta)
=
\int_0^\infty e^{-s}\mathcal BH(\beta s)\,ds
=
\frac{1}{\beta}\int_0^\infty
e^{-\tau/\beta}\mathcal BH(\tau)\,d\tau.
\end{equation}

Both sides are analytic in $\mathcal C_R$, hence by Theorem 3.2.6 in \cite{ablowitz}, we can
extend the equality to $\mathcal C_R$. Since $R$ is arbitrary,
\eqref{eq:H-Borel-Laplace} holds throughout $\mathcal C_\infty$.

Finally, the converse part of Theorem \ref{Watson-Sokal} gives
a factorial remainder estimate for this Borel--Laplace integral
on every $\mathcal C_{R_0}$, where $R_0<R$. Its coefficients are
the coefficients calculated above. Hence \eqref{eq:H-expansion}
is Borel summable to $H$.
}\end{proof}

\section{Remarks on Some Applications}

The results in \cite{main} were established specifically for exponential initial conditions, which represent a special case of the broader class considered in this study. Consequently, the work of \cite{main} can be viewed as a more restrictive instance of our generalized approach, and its expansions are recovered in this current work. Accordingly, we establish the following theorem.

\begin{theorem}
Suppose that $V$ is bounded below, in addition to Assumption \ref{assumptions}(2). For the exponential initial condition $e^{-\beta V}$, $0<\beta\leq1$, both the analytical and the graphical asymptotic expansions of the solution of the Cauchy problem studied in \cite{main} are recovered.
\end{theorem}
\begin{proof}
    We use the same construction of \eqref{eq:smoothed-f} with $P\equiv1$, choosing $\chi=1$ on $[\min\{0,\inf V\}-1,\infty)$ and $\chi=0$ on $(-\infty,\min\{0,\inf V\}-2]$. Then $f(x)=\chi(x)e^{-x}\in\mathcal S(\R)$, $f(\beta V)=e^{-\beta V}$ for $0<\beta\leq1$, and $f^{(n)}(0)=(-1)^n$. We obtain from Theorem \ref{expansion}

    \begin{equation}
        \Phi_t{ ^\beta}(\phi){ \mathrel{\sim}}\sum_{m=0}^\infty \frac{(-1)^m{ \beta^m}}{m!}\langle V^m_\phi\rangle_{\nu_t}, \qquad t\geq0,\quad \phi\in\R^d,
    \end{equation}
and from Theorem \ref{graphexpansion}{ , both as $\beta\downarrow0$,} 

\begin{equation}
    \Phi_t{ ^\beta}(\phi){ \mathrel{\sim}}\sum_{m=0}^\infty \frac{(-1)^m{ \beta^m}}{m!}\sum_{G\in \mathcal F(m)} V[G](t,\phi),\qquad t\geq0,\quad \phi\in\R^d,
\end{equation}

\end{proof}

We now characterize the contributing graphs in the formal logarithm of the solution $\Phi_t{ ^\beta}$ studied in \cite{main}. 

\begin{theorem} The only partition $\pi\in\mathcal P([n])$ giving a nonzero contribution is $\pi=\hat{0}$, i.e, the only graphs that contribute to the logarithmic expansion are connected graphs. 
\end{theorem}

\begin{proof}

Since $f(z)=e^{-z}$ near zero, we have $\log f(z)=-z$, from which
\begin{equation}
\kappa_1=-1, \qquad \kappa_r=0 \quad \text{for all } r\ge2.
\end{equation}

Now let $\pi\in\mathcal P([n])$. If $\pi\neq\hat0$, then $\pi$ contains at least one block
$B$ with $|B|\ge2$. Since $\tilde w(\pi)=\prod_{B\in\pi}\kappa_{|B|}$, the factor
$\kappa_{|B|}$ appears in this product. Hence $\tilde w(\pi)=0$ for every
$\pi\neq\hat0$. Thus the only partition with nonzero contribution is $\pi=\hat0$.

Therefore, by Theorem \ref{logexpgraphthm}
\begin{equation}
\log\Phi_t{ ^\beta}(\phi)
=
\sum_{n\ge1}\frac{(-1)^n{ \beta^n}}{n!}\,\sum_{\substack{G\in\mathcal F(n)\\ \sigma(G)=\hat1}}V[G](t,\phi).
\end{equation}
\end{proof}

\begin{remark}
Borel summability can also be recovered from our current work. Consider the case when $P\equiv1$ in Assumptions \ref{Borel}, we have $H=M$ and
$\mathcal BH=\mathcal BM$ and therefore, Theorem \ref{polyexpborel} reduces to the exponential case proved in \cite[Theorem~2.21]{main}.
\end{remark}

We illustrate now a one-dimensional setting for which the convolution kernel and the corresponding solution can be represented explicitly. In this case, the transition kernel is a Poisson-weighted mixture of translated Gaussian densities. This representation provides both an exact benchmark for the theoretical expansion and a direct algorithm for its numerical evaluation.

Let $\phi, x\in\R$ and assume
\begin{equation}
V(\varphi)=\varphi^2,
\qquad
f_2(x)=(1+x)e^{-x}.
\end{equation}

For applying the preceding results, we take $f=\chi f_2$ with the cutoff in \eqref{eq:smoothed-f}. Since $V\geq0$ and $\beta>0$, this preserves the initial condition and every derivative at zero.

Assume that the symbol $\Psi$ of the kernel convolution $\nu_t$ is of the form 
\begin{equation}
\Psi(-i\xi)
=
-\frac{\sigma^2}{2}\xi^2
+
z\int_{\mathbb{R}\setminus\{0\}}\bigl(e^{iy\xi}-1\bigr)\,d\delta_a(y)
=
-\frac{\sigma^2}{2}\xi^2+z\bigl(e^{ia\xi}-1\bigr),\qquad \xi\in\R,
\end{equation}
where $\delta_a$,$\,a>0$, is the Dirac measure, $z>0$ is the intensity of the jumps, and $\sigma>0$. 

We then have by \eqref{nuhat}
\begin{equation}
\widehat{\nu_t}(\xi)=e^{t\Psi(-i\xi)}=
e^{-zt}\sum_{n=0}^{\infty}\frac{(zt)^n}{n!}\,e^{ian\xi}e^{-\frac{\sigma^2 t}{2}\xi^2}.
\end{equation}

Then, by the inverse Fourier transform, we obtain an explicit expression of the convolution kernel
\begin{equation}
\nu_t(x)=e^{-zt}\sum_{n=0}^{\infty}\frac{(zt)^n}{n!}\,g_{\sigma^2 t}(x-na),
\quad\text{ with }\quad
g_s(x)=\frac{1}{\sqrt{2\pi s}}e^{-x^2/(2s)}{ ,\quad s>0.}
\end{equation}

The exact solution is given by 
\begin{equation}\label{numexp}
\Phi_t^\beta(\varphi)=(\nu_t*f_{\beta})(\varphi),
\qquad
f_{\beta}(\varphi)=(1+\beta\varphi^2)e^{-\beta\varphi^2}.
\end{equation}

{
The exact convolution in \eqref{numexp} can be evaluated from the Gaussian mixture above. Its direct and logarithmic finite-order coefficients are given by Theorems \ref{graphexpansion} and \ref{logexpgraphthm}, respectively.
}

\section*{Acknowledgments}
This work is supported by the Interdisciplinary research Center for Intelligent Secure Sys-
tems at King Fahd University of Petroleum and Minerals under the project $\#IN SS2602$.\\
The authors gratefully acknowledge this support.

\end{document}